\documentclass[preprint,1p]{elsarticle}
\usepackage{graphicx}
\usepackage{amssymb}
\usepackage{amsmath,amsthm}
\usepackage{amsfonts}
\usepackage{graphicx}
\usepackage{graphics}
\usepackage{color}
\usepackage{latexsym}
\usepackage{url}
\def\biglf{\par\bigskip\noindent}
\def\fs#1{\footnotesize#1}
\newcommand{\R}{\ensuremath{\mathbb{R}}}
\newcommand{\D}{\ensuremath{\mathcal{D}}}
\newcommand{\B}{\ensuremath{\mathcal{B}}}
\newcommand{\N}{\ensuremath{\mathbb{N}}}
\newcommand{\Z}{\ensuremath{\mathbb{Z}}}

\def\bfm#1{\protect{\makebox{\boldmath $#1$}}}

\def\a {\bfm{a}}

\def\p {\bfm{p}}

\def\fa {\mbox{ for all }}

\def\wh{\widehat}
\def\wt{\widetilde}
\newtheorem{Theorem}{Theorem}[section]
\newtheorem{Corollary}[Theorem]{Corollary}
\newtheorem{Lemma}[Theorem]{Lemma}

\theoremstyle{definition}
\newtheorem{Definition}[Theorem]{Definition}
\theoremstyle{remark}
\newtheorem{Remark}[Theorem]{Remark}
\numberwithin{equation}{section}

\begin{document}
\begin{frontmatter}
\title{Analysis of moving least squares approximation revisited}%
\author{D. Mirzaei}
\ead{d.mirzaei@sci.ui.ac.ir}%
\address{Department of Mathematics, University of Isfahan, 81745-163 Isfahan, Iran. }

\date{\small\textsl{\today}}
%\dedicatory{}%
%\commby{}%
% ----------------------------------------------------------------
\begin{abstract}
In this article the error estimation of the moving least squares approximation is provided for functions in fractional order Sobolev spaces. The analysis presented in this paper extends the previous estimations and explains some unnoticed mathematical details. An application to Galerkin method for partial differential equations is also supplied.
\end{abstract}
\end{frontmatter}
% ----------------------------------------------------------------
%%%%%%%%%%%%%%%%%%%%%%%%%%%%%%%%%%%%%%%%%%%%%%%%%%%%%%%%%%%%%%%%%%%%%%%%%%%%%%%%%%%%%%%%
\section{Introduction}\label{SectIntro}
The {\em Moving Least Squares (MLS)} approximation was introduced in an early paper by Lancaster and Salkauskas \cite{lancaster-salkauskas:1981-1} in
1981 with special cases going back to McLain \cite{mclain:1974-1,mclain:1976-1} in 1974 and 1976 and to Shepard \cite{shepard:1968-1} in 1968. For other early studies we can
mention the work of Farwig \cite{farwig:1986,farwig:1987,farwig:1991}.
Since, in MLS one writes the value of the unknown function in terms of {\em scattered} data, it can be used as an approximation to span the trial space in
meshless (or meshfree) methods. This approximation has found many applications in curve fitting and numerical solutions of partial differential equations since early nineties \cite{belytschko-et-al:1994-1,belytschko-et-al:1996-1,atluri-zhu:1998-1,mirzaei-schaback:2013-1}.

The error analysis of MLS approximation was provided by some authors, beginning with the work of Farwig \cite{farwig:1991} which is limited
to a univariate case.
%without considering derivatives.
The connection to Backus-Gilbert optimality was studied by Levin \cite{levin:1998-1} in 1998, and later it was used by Wendland \cite{wendland:2001-1,wendland:2000-1,wendland:2005-1} in a more elaborated setting.
In Liu et. al. \cite{belytschko-et-al:1997-1} the analysis is presented for smooth functions in $C^{m+1}(\Omega)\cap H^{m+1}(\Omega)$.
%Their analysis has missed some mathematical details when concerning the points near the boundary.
Armentano and Dur\'{a}n \cite{armentano-duran:2001-1} proved error estimates in $L^\infty$ for the function and its first derivatives in one dimensional case. Afterward Armentano \cite{armentano:2001-1} generalized this to multi-dimensional cases but it is still restricted to ``convex" domains and
Sobolev spaces of order one.
One can also find an estimation in Han and Meng \cite{han-meng:2001-1} for reproducing kernel particle methods (which is related to the MLS approximation) for integer order Sobolev spaces. They assumed a constant bound for the norm of the inverse matrix (matrix $A$ in text) and considered it for special cases in one dimension and first order approximations. Note that the role of this matrix is very crucial in analysis.
The paper of Zuppa \cite{zuppa:2003-1} is also limited to some specific situations.
In Wendland \cite{wendland:2001-1,wendland:2005-1} the analysis presented only for the function in classical function spaces. We can also mention the work of Melenk \cite{melenk:2004-1} where the theoretical and computational aspects of some meshless approximation methods, including MLS, are considered.

The collocation method based on the MLS approximation is called {\em finite point method}.
An analysis for this method has been presented in \cite{cheng-cheng:2008-1}. Besides, an interpolating MLS is
developed recently. For error analysis and applications to element-free Galerkin method
see \cite{ren-et-al:2014-1,wang-et-al:2014-1}.

The present work is based on the theory of Wendland and extends all the above results to a general case. All mathematical details are provided, special care is taken near the boundary, and lower bound for the minimum eigenvalue of the MLS local matrix is derived in general case, independent of the mesh-size. Besides, the analysis is presented for
functions in fractional order Sobolev spaces.
Finally an application to Galerkin methods for elliptic PDEs is investigated.
% which allows functions with less smoothness than those in previous theories in MLS.

%%%%%%%%%%%%%%%%%%%%%%%%%%%%%%%%%%%%%%%%%%%%%%%%%%%%%%%%%%%%%%%%%%%%%%%%%%%%%%%%%%%%%%%%
\section{MLS approximation}\label{SectGMLS}

Let $\Omega\subset \R^d$, for positive integer $d$, be a nonempty and bounded set.
%with Lipschits continuous boundary $\partial \Omega$.
In the next section, more conditions on $\Omega$ will be imposed. Assume,
$$
X = \{x_1, x_2,\ldots ,x_N\}\subset \Omega,
$$
is a set containing $N$ scattered points, called \emph{centers} or \emph{data site}. Distribution of points
should be well enough to pave the way for analysis.

Henceforth, we use $\mathbb P_m^d$, for $m\in\N_{0}=\{n\in\Z, n\geqslant 0\}$,
as the space of $d$-variable polynomials of degree at most $m$ of dimension
$Q={m+d\choose d}$. A basis for this space is denoted by $\{p_1,\ldots,p_Q\}$ or $\{p_\alpha\}_{0\leqslant|\alpha|\leqslant m}$.
As usual, $B(x,r)$ stands for the ball of radius $r$ centered at $x$.

The MLS, as a meshless approximation method, provides an approximation $s_{u,X}$ of $u$ in terms of values $u(x_j)$ at
centers $x_j$ by
\begin{equation}\label{mlsapp}
u(x)\approx s_{u,X} (x) = \sum_{j=1}^N a_j(x) u(x_j), \quad x\in  \Omega,
\end{equation}
where $a_j$ are \emph{MLS shape functions} given by
\begin{equation}\label{shapefunc-aj}
a_j(x)=w(x,x_j)\sum_{k=1}^Q\lambda_k(x) p_k(x_j),
\end{equation}
where the influence of the centers is governed by weight function $w_j(x)=w(x,x_j)$, which vanishes for arguments $x,x_j\in\Omega$ with $\|x-x_j\|_2$ greater than
a certain threshold, say $\delta$. Thus we can define $w_j(x)=\Phi((x-x_j)/\delta)$ where $\Phi:\R^d\to\R$ is a nonnegative function with support in the unit ball $B(0,1)$. Coefficients $\lambda_k(x)$ are the unique solution of
\begin{equation}\label{system-lambdak}
\sum_{k=1}^Q\lambda_k(x)\sum_{j\in J(x)}w_j(x)p_k(x_j)p_\ell(x_j)=p_\ell(x),\quad 0\leqslant \ell\leqslant Q,
\end{equation}
where $J(x)=\{j: \|x-x_j\|_2\leqslant \delta\}$ is the family of indices of points in the support of $w$.
In vector form
\begin{equation*}
\a(x) = W(x)P^T(PW(x)P^T)^{-1}\p(x),
\end{equation*}
where $W(x)$ is the diagonal matrix carrying the weights $w_j(x)$ on its diagonal, $P$ is a $Q\times \#J(x)$ matrix of values $p_k(x_j)$, $j\in J(x)$, $1\leqslant k\leqslant Q$, and $\p=(p_1,\ldots ,p_Q)^T$.
In MLS one finds the best approximation to $u$ at point $x$, out of $\mathbb P_m^d$ with respect to a discrete $\ell^2$ norm induced by
a \emph{moving} inner product, where the corresponding weight function depends not only on points $x_j$ but also on the evaluation point $x$ in question.
Note that $A(x)=PW(x)P^T$ is a symmetric positive definite matrix for all $x\in\Omega$.
More details can be found in Chapter 4 of \cite{wendland:2005-1}.

In what follows we will assume that $\Phi$ is nonnegative and continuous on $\R^d$ and positive on the ball $B(0,1/2)$. In many application we can assume that
$$
\Phi(x)=\phi(\|x\|_2), \quad x\in \R^d,
$$
meaning that $\Phi$ is a radial function. Here $\phi:[0,\infty)\to \R$ is positive on $[0,1/2]$, supported in $[0,1]$ and its even extension is nonnegative and continuous on $\R$.

If, further, $\phi$ is sufficiently smooth, derivatives of $u$ are usually approximated by derivatives of $s_{u,X}$,
\begin{equation}\label{mlsderivapp}
D^\alpha u \approx D^\alpha s_{u,X} (x) = \sum_{j=1}^N D^\alpha a_j(x) u(x_j), \quad x\in  \Omega,
\end{equation}
and they are called \emph{standard derivatives}. They are different from {\em GMLS} or {\em diffuse derivatives} \cite{mirzaei-et-al:2012-1} which are not the aim of this paper.

%%%%%%%%%%%%%%%%%%%%%%%%%%%%%%%%%%%%%%%%%%%%%%%%
\section{Error estimation}\label{Sect-errorbound}
Since error estimates will be established using a variety of Sobolev spaces, we introduce them now. Let $\Omega\subset \R^d$ be a domain. For $k\in\N_0$, and $p\in[1,\infty)$, we define the Sobolev space $W_p^k(\Omega)$ to consist of all $u$ with distributional derivatives $D^\alpha u\in L^p(\Omega)$, $|\alpha|\leqslant k$. The (semi-)norms associated with these spaces are defined as
$$
|u|_{W_p^k(\Omega)}:=\left(\sum_{|\alpha|=k}\|D^\alpha u\|_{L^p(\Omega)}^{p} \right)^{1/p},\;\;
\|u\|_{W_p^k(\Omega)}:=\left(\sum_{|\alpha|\leqslant k}\|D^\alpha u\|_{L^p(\Omega)}^{p} \right)^{1/p}.
$$
The case $p=\infty$ is defined in the standard way
$$
|u|_{W_\infty^k(\Omega)}:=\sup_{|\alpha|=k}\|D^\alpha u\|_{L^\infty(\Omega)}, \;\;
\|u\|_{W_\infty^k(\Omega)}:=\sup_{|\alpha|\leqslant k}\|D^\alpha u\|_{L^\infty(\Omega)}.
$$
For fractional order Sobolev spaces, we use the norms below. Let $p\in[1,\infty)$,
$k\geqslant0$, $k\in \Z$, and let $0 < s < 1$. We define the fractional order Sobolev spaces
$W_p^{k+s}(\Omega)$ to be the space of all $u$ for which the norms below are finite.
\begin{align*}
|u|_{W_p^{k+s}(\Omega)}&:=\left(\sum_{|\alpha|=k}\int_\Omega\int_\Omega \frac{|D^\alpha u(x)-D^\alpha u(y)|^p}{|x-y|^{d+ps}}dxdy \right)^{1/p},\\
\|u\|_{W_p^{k+s}(\Omega)}&:=\left( \|u\|_{W_p^{k}(\Omega)}+|u|_{W_p^{k+s}(\Omega)}\right)^{1/p}.
\end{align*}
The first step in deriving error estimates is to consider only local regions $\D$ that are \emph{star-shaped} with respect to a ball.
A domain $\D\subset \R^d$ is said to be {star-shaped} with respect to a ball $B=B(y,\rho) = \{x\in\R^d : \|x-y\|\leqslant \rho\}$ if for every $x\in \D$,
the closed convex hull of $\{x\}\cup B$ is contained in $\D$.
Let
$$
\rho_{\max} =\sup\{\rho: \D \mbox{ is star-shaped with respect to a
ball of radius } \rho\},
$$
then the \emph{chunkiness parameter} of $\D$ is defined by
$
\gamma = \frac{d_\D}{\rho_{\max}}
$
where $d_\D$ is the diameter of $\D$.

Approximating a function $u\in W^{m+1}_q(\D)$ by averaged Taylor polynomials
$Q_m u  \in\mathbb P_m^d$ is discussed in \cite[Chapter 4]{brenner-scott:2008-1}.
The averaged Taylor polynomials are defined as follows. Let $B$ be a ball with respect to which $\D$ is star-shaped having radius
$\rho\geqslant \frac{1}{2}\rho_{\max}$. Then
$$
Q_m u(x):=\sum_{|\alpha|\leqslant m}\frac{1}{\alpha!}\int_B D^\alpha u(y)(x-y)^\alpha\varphi(y)dy,
$$
where $\varphi(y)\geqslant 0$ is a $C^\infty$ ``bump" function supported in $B$ satisfying both $\int_B\varphi(y)dy=1$ and $\max \varphi \leqslant C\rho^{-d}$.

In \cite{brenner-scott:2008-1} the $W_p^\ell$ bounds on $u-Q_mu$ are given for integer $\ell$ when $u\in W_p^{m+1}(\D)$ and $\ell\leqslant m+1$.
A version of these results
that applies when $u$ belongs to $W_p^{m+s}(\D)$, $0 \leqslant s < 1$, was proved in \cite{narcowich-et-al:2004-1}.
An improvement of conditions (range of $s$) was discussed in \cite{narcowich-et-al:2006-1} by the same authors.
\begin{Lemma}\label{lem-tylor-errLinf}
Let $B$ be a ball in $\D$ such that $\D$ is star-shaped
with respect to $B$ and such that its radius $\rho\geqslant (1/2)\rho_{\max}$. Let $Q_mu$ be
the Taylor polynomial of order $m$ of $u$ averaged over $B$ where $u \in  W^{m+s}_p(\D)$ for $0\leqslant s<1$ and $p\in[1,\infty)$.
Let $m>d/p$ for $p>1$ and $m\geqslant d$ for $p=1$.
Then
there exists constant $C=C(m,d,p,\gamma)$ such that
\begin{equation}\label{tayloravgLinf-frac}
\|u-Q_mu\|_{L^\infty(\D)}\leqslant C\, d_\D^{m+s-d/p}|u|_{W^{m+s}_p(\D)},
\end{equation}
where $d_\D$ is the diameter of $\D$.
\end{Lemma}
We should note that the identity
\begin{equation}\label{QmDa}
D^\alpha Q_mu = Q_{m-|\alpha|}D^\alpha u, \quad \mbox{for all } u\in W^{|\alpha|}_1(\D).
\end{equation}
which is found in \cite[Section 4]{brenner-scott:2008-1}, holds for
$|\alpha|\leqslant m$.
Applying Lemma \ref{lem-tylor-errLinf} on $D^\alpha u$ instead of $u$, using the identity \eqref{QmDa} and the inequality
$|D^\alpha u|_{W_p^{k+s-|\alpha|}(\D)}\leqslant |u|_{W_p^{k+s}(\D)}$, we obtain
\begin{Corollary}\label{cor-taylor-LinfDa}
Let $0\leqslant s<1$. For $u\in W_p^{m+s}(\D)$,
\begin{equation}\label{tayloravgLinfDa-frac}
\|D^\alpha u-D^\alpha Q_mu\|_{L^\infty(\D)}\leqslant C\, d_\D^{m+s-|\alpha|-d/p}|u|_{W^{m+s}_p(\D)},
\end{equation}
provided that $m>|\alpha|+d/p$ for $p>1$ and $m\geqslant |\alpha|+d$ for $p=1$.
\end{Corollary}
Analogous to the error bound \eqref{tayloravgLinf-frac}, one can easily derive the $W_q^\ell$ bound for $u-Q_mu$ when $u\in W_p^{m+s}(\D)$.
We give the results in the following Lemma.
\begin{Lemma}
Let $q\in[1,\infty]$, $p\in [1,\infty)$ and $\alpha$ be a multi-index satisfying $m>|\alpha|+d/p$ for $p>1$ and $m\geqslant |\alpha|+d$ for $p=1$. With the notation and assumptions of Lemma \ref{lem-tylor-errLinf},
we have
\begin{equation}\label{tayloravgLq-frac}
\|u-Q_mu\|_{W_q^{|\alpha|}(\D)}\leqslant C\, d_\D^{m+s-|\alpha|+d(1/q-1/p)}|u|_{W^{m+s}_p(\D)},
\end{equation}
where $C=C(m,d,p,q,\alpha,\gamma)$.
\end{Lemma}
\begin{proof}
Although the proof can be implicitly extracted from \cite{narcowich-et-al:2004-1}, but we present it here for the reader's conveniences.
Let $q\in[1,\infty)$. Using the definition of Sobolev norms, we have
\begin{align*}
\|u-Q_mu\|_{W_q^{|\alpha|}(\D)}^q&=\sum_{|\beta|\leqslant|\alpha|}\int_\D |D^\beta(u-Q_mu)|^qdx \\
&\leqslant \#\{\beta\in\N_0^d:|\beta|\leqslant|\alpha|\}\times \mathrm{vol}(\D)\Big(\max_{|\beta|\leqslant|\alpha|}\|D^\beta(u-Q_mu)\|_{L^\infty(\D)}\Big)^q\\
&\leqslant C(d,\alpha)\, d_\D^{d}\max_{|\beta|\leqslant|\alpha|}\|D^\beta(u-Q_mu)\|_{L^\infty(\D)}^q\\
&\leqslant C(m,d,p,\alpha,\gamma)\,d_\D^{q(m+s-|\alpha|+d(1/q-1/p))}|u|_{W^{m+s}_p(\D)}^q.
\end{align*}
At the third line above, we use the facts that
$\mathrm{vol}(\D)\leqslant C_d d_\D^d$
and
$$
\#\{\beta\in\N_0^d:|\beta|\leqslant|\alpha|\}=\sum_{i=0}^{|\alpha|}{i+d-1 \choose d-1}={|\alpha|+d\choose d} = \mathcal O(|\alpha|^d).
$$
In the last line, Corollary \ref{cor-taylor-LinfDa} has been applied. Finally taking the $q$-th root of the both sides completes the proof
with the new constant $C=C(m,d,p,q,\alpha,\gamma)$.
The case $q=\infty$ can be proved by a similar argument (see also \cite[Proposition 11.29]{wendland:2005-1}).
\end{proof}
\begin{Remark}\label{remark-s1}
In case $s=1$, if we assume $m+1>|\alpha|+d/p$ for $p>1$ and $m+1\geqslant |\alpha|+d$ for $p=1$ then
estimations \eqref{tayloravgLinf-frac} and \eqref{tayloravgLq-frac}
 are still valid, due to \cite{brenner-scott:2008-1}.
The reader should be cautious that these error bounds can not be obtained by inserting $s=0$ and replacing $m$ by $m+1$ in fractional cases, because the later
produces $Q_{m+1}u$.
\end{Remark}
\biglf

Up to this point, we reviewed some Sobolev error bounds for a function which is approximated by the averaged Taylor polynomial on a star-shaped domain. These bounds are usually used for analyzing the finite element method (FEM). Now we turn to the MLS, as a meshless approximation method, and employ the above bounds to analyze it.
The final bound will be presented for functions in fractional Sobolev spaces. Although one can use the interpolation arguments (for example the ``real" method based on K-functionals)
to extend the integer order Sobolev spaces to fractional ones, here we follow the direct approach because all materials are provided via \eqref{tayloravgLinf-frac} and \eqref{tayloravgLq-frac}.

First we introduce some other notations.
For a set of points
$X=\{x_1,x_2,\ldots,x_N\}$ in a bounded domain
$\Omega\subset \R^d$, the {\em fill distance} is defined to be
\begin{equation*}
h_{X,\Omega}=\sup_{x\in\Omega}\min_{1\leqslant j\leqslant N}\|x-x_j\|_2,
\end{equation*}
and the {\em separation distance} is defined by
\begin{equation*}
q_{X}=\frac{1}{2}\min_{i\neq j}\|x_i-x_j\|_2.
\end{equation*}
A set $X$ of data sites is said to be
{\em quasi-uniform} with respect to a constant $c_{\mathrm{qu}}>0$ if
\begin{equation}\label{quasi-uniform}
q_X\leqslant h_{X,\Omega}\leqslant  c_{\mathrm{qu}} q_X.
\end{equation}
A set $X=\{x_1,\ldots,x_N\}\subset \R^d$ with $N\geqslant Q$ is called $\mathbb P_m^d$-unisolvent if the zero polynomial is the only polynomial
from $\mathbb P_m^d$ that vanishes on $X$.

A set $\Omega\subset \mathbb R^d$ is said to satisfy an {\em
interior cone condition} if there exist an angle
$\theta\in(0,\pi/2)$ and a radius $r>0$ such that for every
$x\in\Omega$ a unit vector $\xi(x)$ exists such that the cone
$$
C(x,\xi,\theta,r):=\big\{ x+ty: y\in\R^d, \|y\|_2=1, y^T\xi\geqslant \cos\theta, t\in[0,r]\big\}
$$
is contained in $\Omega$.

 Assuming the compact set $\Omega$ satisfies an interior cone condition with radius $r$ and angle $\theta$, and data site $X\subset \Omega$ satisfies the quasi-uniform condition \eqref{quasi-uniform},  Wendland \cite[Chapter 4]{wendland:2005-1} proved that shape functions $\{a_j(x)\}$ from MLS approximation
\eqref{mlsapp} provide a {\em stable local polynomial reproduction}  of degree $m$ on $\Omega$, i.e.
there exist constants
$h_0,C_{1},C_{2}>0$ independent of $X$ such that for every $x\in\Omega$
\begin{enumerate}
\item $\sum_{j=1}^Na_{j}(x) p(x_j)=p(x),
\,\fa p\in\mathbb P_m^d,$
\item $\sum_{j=1}^N|a_{j}(x)|\leqslant C_{1}$
\item $a_{j}(x)=0\,\, if\,\, \|x-x_j\|_2>\delta=2C_{2}h_{X,\Omega}$,
\end{enumerate}
for all
$X$ with $h_{X,\Omega}\leqslant h_0$.
Constant $C_1$ depends on the weight function $\phi$, and constants $C_2$ and $h_0$ are
\begin{equation}\label{eqC2h0}
C_{2}=\frac{16(1+\sin\theta)^2m^2}{3\sin^2\theta},\quad h_0=\frac{r}{C_{2}}.
\end{equation}
He also proved that, if the weight function possesses $k$ continuous derivatives then the
approximant $s_{u,X}$ is also in $C^k$.
Using the above properties, he proved the error bound
$$
\|u-s_{u,X}\|_{L^\infty(\Omega)}\leqslant Ch_{X,\Omega}^{m+1}|u|_{C^{m+1}(\Omega^*)},
$$
where $|u|_{C^{m+1}(\Omega^*)}:=\max_{|\beta|=m+1}\|D^\beta u\|_{L^\infty(\Omega^*)}$ in which
$\Omega^*=\overline{\cup_{x\in\Omega}B(x,C_{2}h_0)}$ can be obviously larger than the exact domain $\Omega$.
Here $D^\beta$ is the classical derivative operator on space $C^{m+1}$.
The results of the present paper (in a special case) extend this bound for functions $u$ in fractional order Sobolev space
$W_p^{m+s}(\Omega)$.

Taking derivatives of order $\alpha$, $|\alpha|\leqslant m$, under some mild conditions, we can show that functions $\{D^\alpha a_j(x)\}$ in approximation
\eqref{mlsderivapp} form another local polynomial reproduction in the following sense:
there exist constants
$h_0,C_{1,\alpha},C_{2}>0$ independent of $X$ such that for every $x\in\Omega$
\begin{enumerate}
\item $\sum_{j=1}^ND^\alpha a_{j}(x) p(x_j)=D^\alpha p(x),
\,\fa p\in\mathbb P_m^d,$
\item $\sum_{j=1}^N|D^\alpha a_{j}(x)|\leqslant C_{1,\alpha}h_{X,\Omega}^{-|\alpha|},$
\item $D^\alpha a_{j}(x)=0\,\, if\,\, \|x-x_j\|_2>\delta=2C_{2}h_{X,\Omega}$,
\end{enumerate}
for all
$X$ with $h_{X,\Omega}\leqslant h_0$. The first and the last items are immediately followed from the previous local polynomial reproduction system.
But proving item 2 invites more challenges. First we prove the following straightforward result.
\begin{Lemma}\label{lem-boundDw}
Let $X=\{x_1,\ldots,x_N\}\subset\Omega$ has fill distance $h_{X,\Omega}$. Suppose that function $\phi:[0,\infty)\to\R$, is supported in $[0,1]$ and its even extension belongs to $C^{m}(\R)$ for $m\in \N_0$. Then
for $w_j(x)=\phi(\|x-x_j\|_2/\delta)$, $x\in\R^d$ and $|\alpha|\leqslant m$ we have
\begin{equation}\label{boundDw}
|D^\alpha w_j(x)|\leqslant C_\alpha h_{X,\Omega}^{-|\alpha|}, \quad \forall\, x\in \Omega, \quad j=1,\ldots,N,
\end{equation}
provided that $\delta=2C_2h_{X,\Omega}$.
\end{Lemma}
\begin{proof}
Since $\phi$ is a compactly supported and $C^m$ function, derivatives of $\phi$ up to order $m$ are continuous and bounded.
The absolute value of $D^\alpha w_j$, has a bound with a factor $\delta^{-|\alpha|}$ times derivatives of $\phi$. This immediately gives the desired bound for sufficiently small $h_{X,\Omega}$.
\end{proof}
The MLS approximation can be implemented in a more stable fashion, if a shifted and scaled polynomial basis function
is used as a basis for $\mathbb P_m^d$. In this case, we use the basis
\begin{equation}\label{basis-shift}
\left\{\frac{(x-z)^\alpha}{h_{X,\Omega}^{|\alpha|}} \right\}_{0\leqslant |\alpha|\leqslant m},
\end{equation}
where $z$ is fixed and depends on the evaluation point to be considered. If $\wh x$ is the evaluation point, the best result will be obtained if we finally set $z=\wh x$. In fact, MLS uses different bases for each evaluation point.
We can do this, because the formulation of MLS approximation and equations \eqref{shapefunc-aj} and \eqref{system-lambdak} are independent of the choice of basis functions.
Thus the MLS shape functions can be written as
\begin{equation}\label{ajx-shift}
a_j(x)=w_j(x)\sum_{|\alpha|\leqslant m}\lambda_\alpha(x)\frac{(x_j-z)^\alpha}{h_{X,\Omega}^{|\alpha|}},\quad j=1,2,\ldots,N,
\end{equation}
where $\lambda_\alpha(x)$ is obtained by solving the positive definite system
\begin{equation}\label{Alambdab-shift}
A(x)\lambda(x)=\p(x), \quad p_\alpha(x) = \left\{\frac{(x-z)^\alpha}{h_{X,\Omega}^{|\alpha|}} \right\}_{0\leqslant |\alpha|\leqslant m}^T,
\end{equation}
where
\begin{equation}\label{A-shift}
A_{i,k}(x)=\sum_{j=1}^N w_j(x)p_i\left(\frac{x_j-z}{h_{X,\Omega}}\right)
p_k\left(\frac{x_j-z}{h_{X,\Omega}}\right), \quad i,k=1,\ldots,Q.
\end{equation}
Since $\phi$ is supported in the unit ball, we used the summation index $\displaystyle \sum_{j=1}^N$ instead of $\displaystyle\sum_{j\in J(x)}$ in the above formulation. Since the set point $X$ satisfies the quasi uniform condition \eqref{quasi-uniform}, the number $\# J(x)$ of points in $J(x)$ can be bounded independent of $h_{X,\Omega}$ \cite{wendland:2005-1}. In fact, for $x_i, x_k\in B(x,\delta)$ and $x_i\neq x_k$ the balls $B(x_i,q_X)$ and $B(x_k,q_X)$ are disjoint. All of these balls with $x_j\in J(x)$ are contained in the ball $B(x,q_X+\delta)$. It is clear that
$$\displaystyle \mathrm{vol}\Big(\bigcup_{j\in J(x)}B(x_j,q_X)\Big)\leqslant \mathrm{vol}\big(B(x,q_X+\delta)\big),$$
which simply gives
$\# J(x)q_X^d\leqslant (\delta+q_X)^d$. Using the quasi-uniform condition and $\delta=2C_2h_{X,\Omega}$, we have
$$
\# J(x)\leqslant (1+2C_2c_{\mathrm{qu}})^d=:C_{\#}.
$$
\begin{Lemma}\label{lem-boundDA}
If the weight function $\phi$ satisfies the assumptions of Lemma \ref{lem-boundDw}, then for a fixed but arbitrary evaluation point $\wh x\in\Omega$ we have
\begin{equation}\label{boundDA}
|D^\alpha A(\wh x)|\leqslant C_\alpha h_{X,\Omega}^{-|\alpha|}, \quad \forall \, \alpha \mbox{ with }|\alpha|\leqslant m,
\end{equation}
where $C_\alpha$ is a constant matrix independent of $h_{X,\Omega}$.
\end{Lemma}
\begin{proof}
Equation \eqref{A-shift} gives
$$
D^\alpha A_{i,k}(x)=\sum_{j=1}^N D^\alpha w_j(x)p_i\left(\frac{x_j-\wh x}{h_{X,\Omega}}\right)p_k\left(\frac{x_j-\wh x}{h_{X,\Omega}}\right).
$$
Evaluating at $\wh x$, taking absolute value from both sides and using $\|x_j-\wh x\|_2\leqslant\delta=2C_2h_{X,\Omega}$ we obtain
\begin{align*}
|D^\alpha A_{i,k}(\wh x)|\leqslant C\sum_{j=1}^N |D^\alpha w_j(\wh x)|
\leqslant CC_{\#}C_\alpha h_{X,\Omega}^{-|\alpha|}.
\end{align*}
This completes the proof.
\end{proof}
Since $A(x)$ is positive definite for all $x\in\Omega$, all eigenvalues are real and positive. If the basis \eqref{basis-shift}
is employed, we can prove that the smallest eigenvalue of $A(x)$ has a lower bound away from zero and independent of $h_{X,\Omega}$.
Proving this assertion helps us to find a bound for $|D^\alpha A^{-1}(x)|$.
First, recall
\begin{equation}\label{lambdamin}
\lambda_{\mathrm{min}}(A(x)) = \min_{v\in\R^Q\setminus\{0\}}\frac{v^TA(x)v}{v^Tv},
\end{equation}
for symmetric matrix $A(x)$. Since $A(x)$ is also positive definite, we necessarily have $\lambda_{\mathrm{min}}(A(x))>0$.
To bound $\lambda_{\mathrm{min}}$ we follow some parts of Melenk's argument \cite{melenk:2004-1} and the concept of norming sets
presented in \ref{appendix-normingsets}.
\begin{Lemma}\label{lem-lambdamin}
Suppose that the bounded set $\Omega\subset\R^d$ satisfies an interior cone condition with radius $r$ and angle $\theta\in(0,\pi/2)$. Let
$X=\{x_1,\ldots,x_N\}\subseteq\Omega$ has fill distance $h=h_{X,\Omega}$ and satisfies
$
h\leqslant r/C_2=:h_0
$
where $C_2$ is defined in \eqref{eqC2h0}.
Suppose that $\delta=2C_2h$ is the size of supports of the weight functions, and for a fixed but arbitrary $\wh x$ in $\Omega$ the set $\{x_j\in X: j\in J(\wh x)\}$ is $\mathbb P_m^d$-unisolvent. Then there exists constant $C_\lambda$ independent of
$h$ such that
$$
\lambda_{\mathrm{min}}(A(\wh x)) \geqslant C_\lambda >0,
$$
provided that the shifted scaled basis functions \eqref{basis-shift} are employed.
\end{Lemma}
\begin{proof}
Let $v^*\in\R^Q$, $v^*\neq 0$, be a vector at which the minimum in \eqref{lambdamin} occurs for $x=\wh x$. Define
\begin{equation}\label{def-pi}
\pi(x):= \sum_{|\alpha|\leqslant m} v^*_\alpha  \frac{(x-\wh x)^\alpha}{h^{|\alpha|}}.
\end{equation}
Now using \eqref{A-shift} we simply have
$$
{v^*}^TA(\wh x)v^* = \sum_{j=1}^N w_j(\wh x) \pi^2(x_j).
$$
Since $\{x_j:j\in J(\wh x)\}$ is $\mathbb P_m^d$-unisolvent, the functionals $Z=\{\delta_{x_j}: j\in J(\wh x)\}$ form a norming set for $\mathbb P_m^d$, i.e. there exists an injective mapping $T: \mathbb P_m^d\to T(\mathbb P_m^d)\subseteq \R^{\# J(\wh x)}$, where $T(p)=\big(p(x_j)\big)_{j\in J(\wh x)}$
(See \ref{appendix-normingsets}).
Set $Z$ allows us to equip $\mathbb P_m^d$ with an equivalent norm via the operator $T$.
We define the norm on $\R^{\# J(\wh x)}$ by
$$
\|p\|_{2,w}^2:=\sum_{j\in J(\wh x)}w_j(\wh x)p^2(x_j),
$$
and the norm on $\mathbb P_m^d$ by the infinity norm. Using the properties of norming sets and setting $p=\pi$, we have
$$
\|\pi\|_{2,w}\geqslant \frac{1}{\|T^{-1}\|}\|\pi\|_{\infty,B(\wh x,\delta)}.
$$
Obviously, the set $\partial B(\wh x,h)=\{x: \|x-\wh x\|=h\}$ is a subset of $B(\wh x,\delta)$.
Definition of $\pi$ in \eqref{def-pi} ensures that the values of $\pi$ on $\partial B(\wh x,h)$ are independent of $h$, because $h$ will be canceled from the numerators and the denominators. In fact $\|\pi\|_{\infty,\partial B(\wh x,h)}$ is bounded from below by a positive factor times $\|v^*\|_1:=\sum_{|\alpha|\leqslant m}|v_\alpha^*|$, and thus we have
\begin{equation*}
\|\pi\|_{\infty,B(\wh x,\delta)} \geqslant \|\pi\|_{\infty,\partial B(\wh x,h)}\geqslant C_{\pi}\|v^*\|_1\geqslant C_{\pi}\|v^*\|_2,
\end{equation*}
where $C_{\pi}$ is the mentioned factor which is independent of $h$. The last inequality follows from the standard relations between one and two norms in $\R^Q$.

It remains to bound $\|T^{-1}\|$.
By assumptions, $\Omega $ satisfies a cone condition with angle $\theta$ and radius $r$, and $h\leqslant r/C_2$. The later gives  $\delta/2\leqslant r$. Of course the cone condition will be obeyed if we use any radius less than $r$. Thus for every $\wh x\in\Omega$,
there exists a cone $C(\wh x)=C(\wh x,\xi,\theta,\delta/2)\subset\Omega\cap B(\wh x,\delta/2)$, and using Lemma \ref{lem-cone-ball} there exists a closed ball
$$\wt B=B(\wt x,\rho\delta)\subset C(\wh x),\quad \rho=\frac{1}{2}\frac{\sin\theta}{1+\sin\theta}.$$
We are going to prove
\begin{equation}\label{eqpiBBc}
\|p\|_{\infty,B(\wh x,\delta)}\leqslant C_1 \|p\|_{\infty,\wt B} \leqslant C |p(x_k)|, \quad x_k\in \wt B\subset B(\wh x,\delta/2),
\end{equation}
for all $p\in\mathbb P_m^d$. Using Lemma \ref{lem-bernstein}, the first inequality satisfies with $C_1=\big(\frac{2}{\rho}\big)^m$.
To prove the second, let $\|p\|_{\infty,\wt B}=|p(x_M)|$, for $x_M\in \wt B$.
Since ball $\wt B$ itself satisfies an interior cone condition with radius $\rho\delta$ and angle $\pi/3$, Theorem \ref{thm-TTTT}
can be applied provided that $h\leqslant \frac{\rho\delta\sqrt3/2}{4(1+\sqrt3/2)m^2}$. One can easily check that this condition is always satisfied because $\delta=2C_2h=\frac{32h(1+\sin\theta)^2m^2}{3\sin^2\theta}$. The proof of the mentioned Theorem shows that
there exists a point $x_k\in X\cap \wt B$ such that $|p(x_k)|\geqslant \frac{1}{2}|p(x_M)|=\frac{1}{2}\|p\|_{\infty,\wt B}$. Thus in \eqref{eqpiBBc}
the constant $C$ can be chosen as $C_{\theta,m}:=2(\frac{4(1+\sin\theta)}{\sin\theta})^m$. Letting
$$w_{\mathrm{min}}:=\min_{s\in[0,1/2]}\phi(s),$$
we can write
\begin{equation*}\label{eqpiBB}
w_{\mathrm{min}}\|p\|_{\infty,B(\wh x,\delta)}^2 \leqslant C_{\theta,m}^2 w_k |p(x_k)|^2\leqslant C_{\theta,m}^2\sum_{j=1}^Nw_j|p(x_j)|^2=C_{\theta,m}^2\|p\|_{2,w}^2,
\end{equation*}
which immediately gives
$$\|T^{-1}\|\leqslant \frac{C_{\theta,m}}{\sqrt {w_{\mathrm{min}}}}.$$
Summarizing all, we have
$$
\lambda_{\mathrm{min}}(A(\wh x)) = \frac{v^*A(\wh x)v^*}{{v^*}^Tv^*}=\frac{\|\pi\|_{2,w}^2}{\|v^*\|_2^2}
\geqslant \frac{C_{\pi}^2}{\|T^{-1}\|^2}\geqslant
\frac{ w_{\mathrm{min}}C_{\pi}^2}{C_{\theta,m}^2}=:C_{\lambda},
$$
which completes the proof.
\end{proof}
\begin{Remark}
The unisolvency condition in Lemma \ref{lem-lambdamin} is a mild condition, because in most cases if $\Omega$ satisfies a cone condition
then $\Omega\cap B(x,\delta)$, $x\in\Omega$, also satisfies another cone condition and for sufficiently small $h_{X\cap B,\Omega\cap B}$ Theorem \ref{thm-TTTT} ensures the unisolvency.
\end{Remark}
\begin{Remark}
The role of ``shifted" and ``scaled" basis functions \eqref{basis-shift} is crucial to bound $\lambda_{\mathrm{min}}$ away from zero and
independent of $h_{X,\Omega}$. Otherwise, experiments show that $\lambda_{\mathrm{min}}$ tends to zero when $h_{X,\Omega}\to0$.
See section 6 of \cite{mirzaei-et-al:2012-1} for numerical results.
\end{Remark}
\begin{Lemma}\label{lem-boundDAinv}
With the notation and assumptions of Lemmas \ref{lem-lambdamin} and \ref{lem-boundDw}, we have
\begin{equation}\label{boundDAint}
|D^\alpha A^{-1}(\wh x)|\leqslant C_\alpha h^{-|\alpha|}, \quad \forall \, \alpha \mbox{ with }|\alpha|\leqslant m,
\end{equation}
where $C_\alpha$ is a constant matrix independent of $h=h_{X,\Omega}$.
\end{Lemma}
\begin{proof}
Taking the derivative of both sides of the well known relation $A^{-1}(x)A(x)=I$ and evaluating at $\wh x$, we obtain
\begin{equation}\label{DAinv}
D^\alpha A^{-1}(\wh x)=-A^{-1}(\wh x)\sum_{\substack{\beta\leqslant \alpha \\ \beta\neq\alpha}}{\alpha \choose \beta}D^{\alpha-\beta}A(\wh x)D^{\beta}A^{-1}(\wh x).
\end{equation}
Induction on $|\beta|$ can be used to prove the desired result. The first step in induction is considered by choosing $\beta=e_j$, the unit vector with $1$ in $j$-th place. We simply have $D^{e_j}A^{-1}=-A^{-1}(D^{e_j}A)A^{-1}$. Since $|e_j|=1$, equation \eqref{boundDA} yields
$$
|D^{e_j}A^{-1}|\leqslant |A^{-1}|\,C_1h^{-1}|A^{-1}|,
$$
where $C_1$ is a constant matrix. From matrix computations, there exists a constant matrix $C\in \R^{Q\times Q}$ such that $|A^{-1}|\leqslant C\|A^{-1}\|_2$ holds. Since $A$ is symmetric positive definite, we have
\begin{align*}
|D^{e_j}A^{-1}|\leqslant C h^{-1}\|A^{-1}\|_2^2
=\frac{C h^{-1}}{\lambda_{\mathrm{min}}^2(A)}
\leqslant \frac{C}{C_\lambda^2} h^{-1},
\end{align*}
which yields the starting point for induction. Now we suppose that $|D^\beta A^{-1}|\leqslant Ch^{-|\beta|}$ holds for all $\beta\leqslant\alpha$ and $\beta\neq \alpha$. Employing \eqref{boundDA}, equation \eqref{DAinv} gives
\begin{align*}
|D^\alpha A^{-1}|\leqslant |A^{-1}|\sum_{\substack{\beta\leqslant \alpha \\ \beta\neq\alpha}}{\alpha \choose \beta}C_{\alpha,\beta}h^{|\beta|-|\alpha|}C_{\beta}h^{-|\beta|}
\leqslant C_\alpha h^{-|\alpha|},
\end{align*}
which completes the proof.
\end{proof}

\begin{Theorem}\label{thm-lebesgueDa}
The shape functions $a_j$, $j=1,\ldots,N$, from the MLS approximation possess the following stability condition for $|\alpha|\leqslant m$,
\begin{equation*}
\sum_{j=1}^N|D^\alpha a_{j}(x)|\leqslant C_{1,\alpha}h_{X,\Omega}^{-|\alpha|},\quad \forall\, x\in\Omega,
\end{equation*}
where $C_{1,\alpha}$ is independent of data site $X$, provided that all assumptions of Lemmas \ref{lem-lambdamin} and \ref{lem-boundDw} are satisfied.
\end{Theorem}
\begin{proof}
Let $h=h_{X,\Omega}$.
First, from \eqref{Alambdab-shift} we have $\lambda(x)=A^{-1}(x)b(x)$ and thus for a fixed but arbitrary $\wh x\in\Omega$
$$
D^\alpha \lambda(\wh x) = \sum_{\eta\leqslant\alpha}{\alpha\choose\eta}D^{\alpha-\eta}A^{-1}(\wh x)\, D^\eta \p(\wh x),\quad \forall\, \alpha \mbox{ with } |\alpha|\leqslant m.
$$
Since $z=\wh x$, obviously all entries of vector $D^\eta \p(\wh x)$ are zero except the $\eta$-entry which is $\eta! h^{-|\eta|}$, i.e.
$D^\eta \p(\wh x)=\eta! h^{-|\eta|}e_\eta$, where $e_\eta$ is a unit vector with $1$ in $\eta$-th place. Now
using \eqref{boundDAint} we can write for a constant matrix $C_{\alpha,\eta}$
\begin{align*}
|D^\alpha \lambda(\wh x)| \leqslant \sum_{\eta\leqslant\alpha}{\alpha\choose\eta}C_{\alpha,\eta}h^{|\eta|-|\alpha|}e_\eta h^{-|\eta|}
\leqslant C_\alpha h^{-|\alpha|},
\end{align*}
where the vector $C_\alpha$ is a bound for $\displaystyle\sum_{\eta\leqslant \alpha}{\alpha\choose\eta}C_{\alpha,\eta}e_\eta$.
Now, taking the derivatives of both sides of equation \eqref{ajx-shift} one obtains
\begin{align*}
D^\alpha a_j( x) = \sum_{\eta\leqslant \alpha}\left\{  {\alpha \choose \eta}D^{\alpha-\eta}w_j( x)
\sum_{|\beta|\leqslant m} D^\eta \lambda_\beta (x)h^{-|\beta|}(x_j-\wh x)^\beta\right\}.
\end{align*}
Evaluating at $\wh x$, applying the bounds of $D^\eta \lambda_\beta(\wh x)$ and $D^{\alpha-\eta}w_j(\wh x)$ and using the fact that $|x_j-\wh x|^\beta\leqslant h^{|\beta|}$, we finally have
\begin{align*}
|D^\alpha a_j(\wh x)| &\leqslant \sum_{\eta\leqslant \alpha}\left\{  {\alpha \choose \eta}C_{\alpha,\eta}h^{|\eta|-|\alpha|}
\sum_{|\beta|\leqslant m} C_\eta h^{-|\eta|}h^{-|\beta|}h^{|\beta|}\right\}\\
&\leqslant C_\alpha h^{-|\alpha|},
\end{align*}
which completes the proof.
\end{proof}
Theorem  \ref{thm-lebesgueDa} establishes the second property of the local polynomial reproduction system $\{D^\alpha a_j\}$.
This will help us to estimate the error function in MLS approximation.
First we note that a region with a Lipschitz boundary automatically satisfies an interior cone condition. More details can be found in \cite{wloka:1987}.
\begin{Theorem}\label{thm-errorLq}
Suppose that $\Omega\subset\R^d$ is a bounded set with a Lipschitz boundary.
Let $m$ be a positive integer, $0\leqslant s < 1$, $p\in[1,\infty)$, $q\in[1,\infty]$ and let $\alpha$ be a multi-index satisfying
$m>|\alpha|+d/p$ for $p>1$ and $m\geqslant |\alpha|+d$ for $p=1$. If $u\in W^{m+s}_p(\Omega)$,
there exist constants $C > 0$ and $h_0>0$ such that
for all $X=\{x_1,\ldots ,x_N\}\subset \Omega$ with $h_{X,\Omega}\leqslant \min\{h_0,1\}$ which are quasi-uniform with the same $c_{\mathrm{qu}}$ in (\ref{quasi-uniform}), the estimate
\begin{align}\label{LPRerrorLp}
\big\|u-s_{u,X} \big\|_{W^{|\alpha|}_q(\Omega)} &\leqslant Ch_{X,\Omega}^{m+s-|\alpha|-d(1/p-1/q)_{+}}\|u\|_{W^{m+s}_p(\Omega)},
\end{align}
holds. Here $(x)_+=\max\{x,0\}$ and $s_{u,X}$ is the MLS approximation of $u$ on data site $X$ in which the corresponding weight function satisfies the assumptions of Lemma
\ref{lem-boundDw}, and the shifted scaled basis polynomials \eqref{basis-shift} are employed.
\end{Theorem}
\begin{proof}
Since $\Omega$ is bounded and has a Lipschitz boundary, we can use the continuous extension operator
$$E_\Omega:W_p^{m+s}(\Omega)\to W_p^{m+s}(\R^d),\quad 1\leqslant p< \infty,$$
to extend any $u\in W^{m+s}_p(\Omega)$ to a function
$v:=E_\Omega u\in W^{m+s}_p(\R^d)$, with $v|_\Omega= u$. Since the extension is continuous, we have
\begin{equation}\label{eq-normequiv}
\|v\|_{W_p^{m+s}(\R^d)}\leqslant C \|u\|_{W_p^{m+s}(\Omega)}.
\end{equation}
The case $s=0$ was constructed by Stein \cite{stein:1971} and works also for $p=\infty$. DeVore and Sharpley \cite{devore-sharpley:1993-1}
have proved this extension for the fractional order spaces.

First we prove \eqref{LPRerrorLp} for $q\in[1,\infty)$. The case $q=\infty$ will be discussed later.
Let the Lipschitz domain $\Omega$ satisfies a cone condition with angle $\theta$ and radius $r$. Assuming $h_0=r/C_2$ in \eqref{eqC2h0},
we first bound the error over subdomains $\mathcal B_k=B(x_k,\delta)\cap\Omega$, $k=1,\ldots,N$, for $\delta=2C_2h_{X,\Omega}$ where $h_{X,\Omega}\leqslant \min\{h_0,1\}$. At the end, we will extend the error bound over entire $\Omega$.
Let $\D_k=B(x_k,2\delta)$, $k=1,\ldots,N$. Clearly, $\D_k\nsubseteq\Omega$ in general. But $\D_k$ is star-shaped with respect to a ball $\wt B\subset \D_k$ with chunkiness parameter $\gamma=2$. Now let $p=Q_{m}v\in\mathbb P_m^d$ be the Taylor polynomial of degree $m$ of $v$ on $\D_k$ averaged over $\wt B$.
The reader should care about the letter $p$, which has been employed for both polynomial and Sobolev notations.
Using the properties of the stable local polynomial reproduction $\{a_j\}$, we can write for $x\in \mathcal B_k$
$$
u(x) -s_{u,X}(x)=u(x) - p(x) + \sum_{j=1}^N a_j(x)\big(p(x_j)-u(x_j)\big),
$$
and in $W_q^{|\alpha|}$ norm,
\begin{equation}\label{equsWB_k}
\|u - s_{u,X} \|_{W_q^{|\alpha|}(\B_k)} \leqslant \|u -  p \|_{W_q^{|\alpha|}(\B_k)}
+\Big\| \sum_{j=1}^N a_{j}(\cdot)\big(p(x_j)-u(x_j)\big)  \Big\|_{W_q^{|\alpha|}(\B_k)}.
\end{equation}
Using the facts that $x_j\in\D_k$ and
$v|_\Omega=u$, the second norm on the right-hand side can be bounded as below
\begin{align*}
\Big\| \sum_{j=1}^N a_{j}(\cdot)\big(p(x_j)-u(x_j)\big)  \Big\|_{W_q^{|\alpha|}(\B_k)}^q
&=\sum_{|\beta|\leqslant|\alpha|}\Big\| \sum_{j=1}^N D^\beta a_{j}(\cdot)\big(p(x_j)-u(x_j)\big)  \Big\|_{L^q(\B_k)}^q\\
&\leqslant  Cd_{\B_k}^{d}\|v-p\|_{L^\infty(\D_k)}^q\sum_{|\beta|\leqslant|\alpha|}\Big(\max_{x\in\B_k}\sum_{j=1}^N \big|D^\beta a_{j}(x)\big|\Big)^{q}\\
&\leqslant Cd_{\B_k}^{d}\|v-p\|_{L^\infty(\D_k)}^q\sum_{|\beta|\leqslant|\alpha|}C_{1,\beta}h_{X,\Omega}^{-|\beta|q}\\
&\leqslant C_\alpha d_{\B_k}^{d} h_{X,\Omega}^{-|\alpha|q} \|v-p\|_{L^\infty(\mathcal D_k)}^q,
\end{align*}
where we use
\begin{align*}
\Big\| \sum_{j=1}^N D^\beta a_{j}(\cdot)\big(p(x_j)-u(x_j)\big)  \Big\|_{L^q(\B_k)}^q &\leqslant \int_{\B_k}\Big(\sum_{j=1}^N \big|D^\beta a_{j}(x)\big|\big|p(x_j)-u(x_j)\big|\Big)^q dx\\
&\leqslant \|v-p\|_{L^\infty(\mathcal D_k)}^q\Big(\max_{x\in\B_k}\sum_{j=1}^N \big|D^\beta a_{j}(x)\big|\Big)^q \int_{\B_k} dx,
\end{align*}
which together with $\int_{\B_k} dx=\mathrm{vol}(\B_k)\leqslant c\,d_{\B_k}^d$ gives the inequality in the second line. The inequality in the third line
follows from Theorem \ref{thm-lebesgueDa}. The last estimate satisfies because $h_{X,\Omega}\leqslant 1$ and $|\alpha|\geqslant |\beta|$.
Thus from \eqref{equsWB_k} we can write
\begin{equation*} %\label{LpnormEq}
\|u - s_{u,X} \|_{W_q^{|\alpha|}(\B_k)}
\leqslant \|u-p\|_{W_q^{|\alpha|}(\B_k)} + C_{\alpha} d_{\B_k}^{d/q} h_{X,\Omega}^{-|\alpha|} \|v-p\|_{L^\infty(\mathcal D_k)}.
\end{equation*}
To bound the both terms on the right-hand side of inequality above, first by \eqref{tayloravgLinf-frac} we have
\begin{align*}
\|v - p\|_{L^\infty(\D_k)} \leqslant c\, d_{\D_k}^{m+s-d/p}|v|_{W_p^{m+s}(\D_k)}.
\end{align*}
Then, since $\B_k\subset \D_k$,
\eqref{tayloravgLq-frac} leads to
\begin{align*}
\|u - p \|_{W_q^{|\alpha|}(\B_k)} &\leqslant \|v - p \|_{W_q^{|\alpha|}(\D_k)}\\
&\leqslant  C\, d_{\D_k}^{m+s-|\alpha|+d(1/q-1/p)}|v|_{W^{m+s}_p(\D_k)}.
\end{align*}
If we assemble everything up to this point and use the facts that $d_{\B_k} \leqslant 2\delta$ and $d_{\D_k}= 4\delta$ we get
\begin{equation}\label{eqDaDsB_k}
\|u - s_{u,X} \|_{W_q^{|\alpha|}(\B_k)} \leqslant C\,h_{X,\Omega}^{m+s-|\alpha|+d(1/q-1/p)}|v|_{W_p^{m+s}(\D_k)}.
\end{equation}
Now we should extend this bound over entire $\Omega$. Since $\delta=2C_2h_{X,\Omega}$ and
$C_{2}\geqslant 1/2$, for every $x\in\Omega$ there is a center $x_j\in B(x,\delta)\cap\Omega$. This clearly shows
$\Omega=\cup_{k=1}^N \B_k \subset \cup_{k=1}^N \D_k=:\Omega^*$.
First, since $\D_k\subset\Omega^*$ we have
\begin{align*}
\sum_{k=1}^N|v|_{W_p^{m+s}(\D_k)}^p&=\sum_{k=1}^N \sum_{|\beta|=m}\int_{\D_k}\int_{\D_k}\frac{|D^\beta u(x)-D^\beta u(y)|^p}{|x-y|^{d+ps}}dxdy\\
&\leqslant\sum_{k=1}^N \sum_{|\beta|=m}\int_{\D_k}\int_{\Omega^*}\frac{|D^\beta u(x)-D^\beta u(y)|^p}{|x-y|^{d+ps}}dxdy\\
&=\sum_{|\beta|=m}\int_{\Omega^*}\left(\sum_{k=1}^N \chi_{\D_k}(x)\right)\int_{\Omega^*}\frac{|D^\beta u(x)-D^\beta u(y)|^p}{|x-y|^{d+ps}}dxdy,
\end{align*}
where $\chi_\D$ denotes the characteristic function of the set $\D$. Note that $n(x):=\sum_{k=1}^N \chi_{\D_k}(x)$ is the number of subdomains $\D_k$
containing $x$. This function can be bounded by a constant because $X$ is a quasi-uniform set. In fact $n(x)$ is the number of points $x_k$ located in the ball $B(x,2\delta)$. Since this ball is contained in a cube of side-length $4\delta/\sqrt d$, we can write
$$
n(x)\leqslant \left(\frac{4\delta}{\sqrt d q_X}\right)^d\leqslant \left(\frac{8C_2h_{X,\Omega}c_{\mathrm{qu}}}{\sqrt d h_{X,\Omega}}\right)^d
=\left(\frac{8c_{\mathrm{qu}}C_2}{\sqrt d}\right)^d.
$$
Thus we have
\begin{align*}
\sum_{k=1}^N|v|_{W_p^{m+s}(\D_k)}^p&\leqslant C
\sum_{|\beta|=m}\int_{\Omega^*}\int_{\Omega^*}\frac{|D^\beta u(x)-D^\beta u(y)|^p}{|x-y|^{d+ps}}dxdy\\
&=C |v|_{W_p^{m+s}(\Omega^*)}^p.
\end{align*}
Now applying \eqref{eqDaDsB_k} and the above bound we can write
\begin{align*}
\|u - s_{u,X} \|_{W_q^{|\alpha|}(\Omega)} &\leqslant \left(\sum_{k=1}^N\|u - s_{u,X} \|_{W_q^{|\alpha|}(\B_k)}^q \right)^{1/q}\\
&\leqslant C\,h_{X,\Omega}^{m+s-|\alpha|+d(1/q-1/p)}\left(\sum_{k=1}^N|v|_{W_p^{m+s}(\D_k)}^q\right)^{1/q}\\
&\leqslant C\,h_{X,\Omega}^{m+s-|\alpha|+d(1/q-1/p)}N^{(1/q-1/p)_{+}}\left(\sum_{k=1}^N|v|_{W_p^{m+s}(\D_k)}^p\right)^{1/p}\\
&\leqslant C\,h_{X,\Omega}^{m+s-|\alpha|+d(1/q-1/p)}h_{X,\Omega}^{-d(1/q-1/p)_{+}}|v|_{W_p^{m+s}(\Omega^*)}\\
&\leqslant C\,h_{X,\Omega}^{m+s-|\alpha|-d(1/p-1/q)_{+}}|v|_{W_p^{m+s}(\Omega^*)}\\
&\leqslant C\,h_{X,\Omega}^{m+s-|\alpha|-d(1/p-1/q)_{+}}|v|_{W_p^{m+s}(\R^d)}\\
&\leqslant C\,h_{X,\Omega}^{m+s-|\alpha|-d(1/p-1/q)_{+}}\|v\|_{W_p^{m+s}(\R^d)}.
\end{align*}
The bound on the third line above follows from standard inequalities relating $p$ and $q$ norms on finite dimensional spaces where
$(x)_{+}=\max\{x,0\}$.
In the fourth line, to bound $N$ by the fill distance,
let $d_\Omega$ be the diameter of $\Omega$. Since $\Omega$ is bounded, there exists a cube of side length $d_{\Omega}/\sqrt d$ that contains $\Omega$.
Thus
$$
N\leqslant \Big(\frac{d_\Omega}{\sqrt d q_X}\Big)^d \leqslant \Big(\frac{c_{\mathrm{qu}}d_\Omega}{\sqrt d h_{X,\Omega}}\Big)^d = ch_{X,\Omega}^{-d}.
$$
In the fifth line, we have used the identity $d\big(1/q-1/p\big)-d\big(1/q-1/p\big)_{+}=-d\big(1/p-1/q\big)_{+}$.
Finally, we invoke the norm equivalence property \eqref{eq-normequiv} to get the final bound
\begin{equation*}
\|u - s_{u,X} \|_{W_q^{|\alpha|}(\Omega)} \leqslant C\,h_{X,\Omega}^{m+s-|\alpha|-d(1/p-1/q)_{+}}\|u\|_{W_p^{m+s}(\Omega)}.
\end{equation*}
The case $q=\infty$ can be proved in a similar way, because
\eqref{tayloravgLq-frac} can be used for $q=\infty$ to bound the first term in \eqref{equsWB_k}, and the second term can be simply bounded by
$$
\Big\| \sum_{j=1}^N a_{j}(\cdot)\big(p(x_j)-u(x_j)\big)  \Big\|_{W_\infty^{|\alpha|}(\B_k)}\leqslant
Ch_{X,\Omega}^{-|\alpha|} \|v-p\|_{L^\infty(\mathcal D_k)}.
$$
The reader can continue the proof to get
\begin{equation*}
\|u - s_{u,X} \|_{W_\infty^{|\alpha|}(\Omega)} \leqslant C\,h_{X,\Omega}^{m+s-|\alpha|-d/p}\|u\|_{W_p^{m+s}(\Omega)}.
\end{equation*}
\end{proof}
\begin{Remark}
According to Remark \ref{remark-s1}, one can easily proceed with the proof of Theorem \ref{thm-errorLq} (by doing some modifications) to get the estimation \begin{equation}\label{errorLp_s1}
\big\|u-s_{u,X} \big\|_{W^{|\alpha|}_q(\Omega)} \leqslant Ch_{X,\Omega}^{m+1-|\alpha|-d(1/p-1/q)_{+}}\|u\|_{W^{m+1}_p(\Omega)},
\end{equation}
provided that $m+1>|\alpha|+d/p$ for $p>1$ and $m+1\geqslant |\alpha|+d$ for $p=1$.
\end{Remark}
\section{Application to Galerkin method for PDEs}
As an application,  we consider the second order elliptic partial differential equation
\begin{align}
-\sum_{i,j=1}^d \frac{\partial}{\partial x_i} \left(\kappa_{ij}\frac{\partial u}{\partial x_j}\right)(x)+c(x)u(x)&\,=\, f(x), \quad x\in \Omega, \label{pde1}\\
\sum_{i,j=1}^d \kappa_{ij}(x)\frac{\partial u}{\partial x_j}(x)n_i(x) + b(x)u(x) & \,=\, g(x),\quad x\in\partial\Omega,\label{pde-bound}
\end{align}
where $\Omega$ is a bounded domain with Lipschitz boundary $\partial \Omega$, and $\kappa_{ij}, c\in L^\infty(\Omega)$, $f\in L^2(\Omega)$, $a_{ij},b\in L^\infty(\partial\Omega)$, $g\in L^2(\partial\Omega)$ and $n$ is the unit normal vector to the boundary $\partial \Omega$. Matrix
$K(x)=\big(\kappa_{ij}(x)\big)$ is assumed to be uniformly elliptic in $\Omega$, i.e. there exists a constant $\gamma$ such that for all $x\in\Omega$
and all $\alpha\in\R^d$ we have $\alpha^TK(x)\alpha\geqslant \gamma\|\alpha\|_2^2$. Moreover, we assume $c\geqslant 0$ and $b\geqslant 0$, and at least
one of them is uniformly bounded away from zero on a subset of nonzero measure on $\Omega$ or $\partial \Omega$, respectively. Under these assumptions
the weak form of equation \eqref{pde1} together with boundary condition \eqref{pde-bound} is $a(u,v)=\ell(v)$ where
$a(u,v):W_2^1(\Omega)\times W_2^1(\Omega)\to \R$
is a coercive and continuous bilinear form defined by
\begin{equation*}%\label{bilinear}
a(u,v)=\int_{\Omega}\left(  \sum_{i,j=1}^d \kappa_{ij}\frac{\partial u}{\partial x_j}\frac{\partial v}{\partial x_i} + cuv \right)d\Omega
+\int_{\partial\Omega}buv\, d\Gamma,
\end{equation*}
and $\ell: W_2^1(\Omega)\to \R$ is a continuous linear functional defined by
\begin{equation*}%\label{bilinear}
\ell(v)=\int_{\Omega}fv\,d\Omega
+\int_{\partial\Omega}gv\, d\Gamma.
\end{equation*}
Using the Lax-Milgram theory, the corresponding variational problem
\begin{equation}\label{varprob}
\mbox{find } u\in W_2^1(\Omega)\mbox{ such that } a(u,v)=\ell(v), \mbox{ for all }v\in W_2^1(\Omega),
\end{equation}
admits a unique solution $u$ and the solution is continuously depended on data $\ell$. This problem has been analyzed in \cite{wendland:1999-1} using radial basis functions interpolation.

To find the numerical solution we use the same Galerkin method as in the
classical finite element method. The approximation solution is sought in a subspace generated by MLS shape functions.
We define for quasi-uniform set $X=\{x_1,\ldots,x_N\}\subset\Omega$
$$
V_N=\mathrm{span}\{a_1,a_2,\ldots,a_N\}
$$
as a subspace of $W_2^1(\Omega)$ and solve the discretized problem
\begin{equation}\label{varprobdisc}
\mbox{find } u_N\in V_N\mbox{ such that } a(u_N,v)=\ell(v), \mbox{ for all }v\in V_N.
\end{equation}
Of course this step concerns the computation of domain and boundary integrals, which is the most difficult stage of the procedure.
But we assume that all integrals are computed accurately and seek a bound for the error $\|u-u_N\|_{W_2^1(\Omega)}$ for the function $u\in W_2^{m+s}(\Omega)$ where $m>1+d/2$ and $0\leqslant s <1$. Our analysis allows to consider functions that are less smooth than the functions in $W_2^{m+1}(\Omega)$.
First, recalling the Cea's Lemma we have
\begin{equation*}
\|u-u_N\|_{W_2^1(\Omega)}\leqslant C \inf_{v\in V_N}\|u-v\|_{W_2^1(\Omega)},
\end{equation*}
where $C$ is a generic constant. Since $s_{u,X}\in V_N$, we obtain
\begin{equation*}
\|u-u_N\|_{W_2^1(\Omega)}\leqslant C \,\|u-s_{u,X}\|_{W_2^1(\Omega)},
\end{equation*}
which leads to the following corollary.
\begin{Corollary}\label{cor-galekin-err}
Let $\Omega\subset\R^d$ be a bounded domain with Lipschitz boundary, and
$m$ be an integer satisfying $m>1+d/2$ and let $s\in[0,1)$.
Suppose that $u\in W_2^{m+s}(\Omega)$ is the solution to the variational problem \eqref{varprob} and $u_N\in V_N$ is the solution of
discretized problem \eqref{varprobdisc} where $V_N$ is constructed by the quasi-uniform set $X=\{x_1,\ldots,x_N\}\subset\Omega$, the weight function $\phi$ satisfying assumptions of Lemma \ref{lem-boundDw}, and the basis functions \eqref{basis-shift}. Then there exist
constants $C$ and $h_0$ such that for all set $X$ with $h_{X,\Omega}\leqslant \min\{h_0,1\}$
the estimation
\begin{equation*}
\|u-u_N\|_{W_2^1(\Omega)}\leqslant C h_{X,\Omega}^{m+s-1}\|u\|_{W_2^{m+s}(\Omega)}
\end{equation*}
holds.
\end{Corollary}

Finally according to \eqref{errorLp_s1} and discussions before Corollary \ref{cor-galekin-err}, if integer $m$ satisfies $m>d/2$ and $u\in W_2^{m+1}(\Omega)$ then the error bound
\begin{equation*}
\|u-u_N\|_{W_2^1(\Omega)}\leqslant C h_{X,\Omega}^{m}\|u\|_{W_2^{m+1}(\Omega)}
\end{equation*}
holds.

The orders are the same as those for classical finite elements. In both cases we can use the technique of Nitsche to estimate the error in $L^2$-norm.
%%%%%%%%%%%%%%%%%%%%%%%
\section{Numerical examples}
Since there are extensive numerical examples in literature, here we will restrict ourselves to a couple of examples, in which we will concentrate on the predicted orders
of the errors in \eqref{LPRerrorLp} and \eqref{errorLp_s1}.

We consider the following example
$$
u(x) = \|x\|_2^\lambda, \quad x\in \Omega\subset\R^d,
$$
where $\lambda$ is a real parameter and $\Omega$ is a bounded region around the origin. It is well known that
$$
u\in W^\tau_p(\Omega) \Longleftrightarrow \lambda > \tau-d/p.
$$
We let $p=2$ and $\Omega=[-0.5,0.5]^2\subset\R^2$, and we assign two values $1.5$ and $3$ to $\lambda$.
According to the theory, in the first case we set $m=2$ and examine \eqref{LPRerrorLp},
and in the second case we set $m=3$ and examine \eqref{errorLp_s1}.
In both cases a regular mesh distribution with the fill distance $h$ is used as a set of centers, the compactly supported $C^4$ Wendland's function $\phi(r) = (1-r)_{+}^6(35r^2+18r+3)$ is employed as a weight function, and $\delta=2mh$ is used as a support-size. Results are presented in Tables \ref{tb1} and \ref{tb2} for $q=2,\infty$, and different order derivatives $\alpha$. The $L_2$-errors are computed using a $(200\times 200)$-point Gauss-Legendre quadrature, and $L_\infty$-errors are computed on a very fine regular mesh of size $h_s=0.005$.

%%%%%%    Table 1     %%%%%%%%%%%%%%%%%%%%%%%%%%
\begin{table}
\centering
\caption{Orders for $\lambda=1.5$ and $m=2$}\label{tb1}
\begin{tabular}{|lllllllll|}
  \hline
  && $L_2$ &&  && $L_\infty$ && \\
  \cline{3-5} \cline{7-9}
  \fs$h$      && \fs{$\alpha=(0,0)$}&&\fs{$\alpha=(1,0)$}  && \fs{$\alpha=(0,0)$} && \fs{$\alpha=(1,0)$} \\
  \hline
  \fs$0.1$    && \fs$-$    && \fs$-$   && \fs$-$     &&  \fs$-$ \\
  \fs$0.05$   && \fs$2.56$ && \fs$1.51$&& \fs$1.50$  &&  \fs$0.50$  \\
  \fs$0.025$  && \fs$2.52$ && \fs$1.50$&& \fs$1.50$  &&  \fs$0.50$  \\
  \fs$0.0125$ && \fs$2.56$ && \fs$1.49$&& \fs$1.50$  &&  \fs$0.62$  \\
 \hline
 \fs{Theory} && \fs$2.5$     && \fs$1.5$   && \fs$1.5$  &&  \fs$0.5$\\
 \hline
\end{tabular}
\end{table}
%%%%%%%%%%%%%%%%%%%%%%%%%%%%%%%%%%%%%%%%%%%%%%%%%

%%%%%%    Table 2     %%%%%%%%%%%%%%%%%%%%%%%%%%
\begin{table}
\centering
\caption{Orders for $\lambda=3$ and $m=3$}\label{tb2}
\begin{tabular}{|lllllllllllll|}
  \hline
              && $L_2$     &&           &&           && $L_\infty$&&            &&\\
  \cline{3-7} \cline{9-13}
  \fs$h$      &&\fs{$\alpha=(0,0)$}&&\fs{$\alpha=(1,0)$}&&\fs{$\alpha=(2,0)$}&&\fs{$\alpha=(0,0)$}&& \fs{$\alpha=(1,0)$}&&\fs{$\alpha=(2,0)$} \\
  \hline
  \fs$0.1$    && \fs$-$    && \fs$-$    && \fs$-$    &&  \fs$-$   && \fs$-$    &&\fs$-$ \\
  \fs$0.05$   && \fs$3.75$ && \fs$3.08$ && \fs$2.06$ &&  \fs$3.00$&& \fs$2.03$ &&\fs$1.00$ \\
  \fs$0.025$  && \fs$3.86$ && \fs$3.01$ && \fs$1.99$ &&  \fs$3.00$&& \fs$2.00$ &&\fs$1.00$ \\
  \fs$0.0125$ && \fs$3.88$ && \fs$3.00$ && \fs$1.93$ &&  \fs$3.00$&& \fs$2.06$ &&\fs$1.00$ \\
 \hline
 \fs{Theory}  && \fs$4$    && \fs$3$    && \fs$2$    &&  \fs$3$   &&  \fs$2$   &&\fs$1$ \\
 \hline
\end{tabular}
\end{table}

As we can see, the experimental results confirm the theoretical bounds.
%%%%%%%%%%%%%%%%%%%%%%%%%%%%%%%%%%%%%%%%%%%%%%%%%

%%%%%%%%%%%%%%%%%%%%%
\section{Appendix}
\begin{appendix}
We restate the following definition, lemmas and theorem from Chapter 3 of the book \cite{wendland:2005-1}.
\begin{Definition}\label{appendix-normingsets}
Let $V$ be a finite dimensional vector space with norm $\|\cdot\|_V$ and let $Z\in V^*$ (the dual space of $V$) be a finite set
consisting of $N$ functionals. We will say that $Z$ is a norming set for $V$ if the mapping $T:V\to T(V)\subseteq\R^N$ defined by
$T(v)=(z(v))_{z\in Z}$ is injective. $T$ is called the sampling operator.
\end{Definition}
If $Z$ is a norming set for $V$, then $T^{-1}:T(V)\to V$ exists and we can simply show that
$$
\|Tv\|_{\R^N}\leqslant \|T\|\, \|v\|_V, \quad \|v\|_V\leqslant \|T^{-1}\|\,\|Tv\|_{\R^N},
$$
which means that $\|\cdot\|_V$ and $\|T(\cdot)\|_{\R^N}$ are equivalent norms on $V$.

\begin{Lemma}\label{lem-cone-ball}
Suppose that $C=C(x,\xi,\theta,r)$ is a cone. Then for every $0<h\leqslant r/(1+\sin\theta)$ the closed ball $B=B(y,h\sin\theta)$ with center
$y=x+h\xi$ and radius $h\sin\theta$ is contained in $C(x,\xi,\theta,r)$.
\end{Lemma}
%The proof of the following Theorem can be found in \cite[Theorem 3.8]{wendland:2005-1}.
\begin{Theorem}\label{thm-TTTT}
Suppose that $\Omega\subset\R^d$ is compact and satisfies an interior cone condition with radius $r>0$ and angle $\theta\in(0,\pi/2)$. Let $m\in\N$ be fixed. Suppose $h>0$ and the set $X=\{x_1,x_2,\ldots,x_N\}\subseteq\Omega$ satisfy
\begin{itemize}
\item[(1)] $h\leqslant \frac{r\sin\theta}{4(1+\sin\theta)^2m^2}$,
\item[(2)] for every $B(x,h)\subseteq\Omega$ there is a center $x_j\in X\cap B(x,h)$;
\end{itemize}
then $Z=\{\delta_{x_1},\ldots,\delta_{x_N}\}$ is a normig set for $\mathbb P_m^d\big|_\Omega$ and the inverse of associated sampling operator is bounded by $2$. In fact for every $p\in\mathbb P_m^d$ there exists $x_k\in \Omega\cap X$ such that $|p(x_k)|\geqslant \frac{1}{2}\|p\|_{\infty,\Omega}$.
If $h=h_{X,\Omega}$, the second item is automatically satisfied.

Note that, the functionals $Z=\{\delta_{x_1},\ldots,\delta_{x_N}\}$ form a normig set for $\mathbb P_m^d$
if and only if $X$ is $\mathbb P_m^d$-unisolvent.
\end{Theorem}
The following Bernstein inequality can be easily proved by using the one dimensional Bernstein inequality
$$
\|p\|_{\infty,(-\rho,\rho)}\leqslant \rho^m \|p\|_{\infty,(-1,1)},\quad \forall\, p\in \mathbb P_m^1.
$$
Details of the proof can be found in \cite[Lemma B.4]{melenk:2004-1}.
\begin{Lemma}\label{lem-bernstein}
Assume that $B_1$ and $B_2$ are two balls of radius $\rho_1$ and $\rho_2$, respectively, and $B_1\subset B_2\subset \R^d$. Then
$$
\|p\|_{\infty,B_2}\leqslant \left(\frac{2\rho_2}{\rho_1} \right)^m \|p\|_{\infty,B_1}, \quad \forall\, p\in \mathbb P_m^d.
$$
\end{Lemma}

\end{appendix}

\section*{Acknowledgment}
Special thanks go to Dr. Keivan Mohajer for careful proofreading, and to the referees
for their comments and suggestions.
%%%%%%%%%%%%%%%%%%%%%%%%%%%%%%%%%%%%%%%%%%%%%%
% ------------------------------------------------ ----------------
%\bibliography{dmref}

\begin{thebibliography}{10}
\expandafter\ifx\csname url\endcsname\relax
  \def\url#1{\texttt{#1}}\fi
\expandafter\ifx\csname urlprefix\endcsname\relax\def\urlprefix{URL }\fi
\expandafter\ifx\csname href\endcsname\relax
  \def\href#1#2{#2} \def\path#1{#1}\fi

\bibitem{lancaster-salkauskas:1981-1}
P.~Lancaster, K.~Salkauskas, Surfaces generated by moving least squares
  methods, Mathematics of Computation 37 (1981) 141--158.

\bibitem{mclain:1974-1}
D.~McLain, Drawing contours from arbitrary data points, Computer Journal 17
  (1974) 318--324.

\bibitem{mclain:1976-1}
D.~McLain, Two dimensional interpolation from random data, Computer Journal 19
  (1976) 178--181.

\bibitem{shepard:1968-1}
D.~Shepard, A two-dimensional interpolation function for irregularly-spaced
  data, in: Proceedings of the 23th National Conference ACM, 1968, pp.
  517--523.

\bibitem{farwig:1986}
R.~Farwig, Multivariate interpolation of orbitrary spaced data by moving least
  squares methods, Journal of Computational and Applied Mathematics 16 (1986)
  79--83.

\bibitem{farwig:1987}
R.~Farwig, Multivariate interpolation of scattered data by moving least squares
  methods, in: Algorithms for Approximation, Clarendon Press, Oxford, 1987, pp.
  193--211.

\bibitem{farwig:1991}
R.~Farwig, Rate of convergence of moving least squares interpolation methods:
  the univariate case, in: Progress in Approximation Theory, Academic Press,
  Boston, 1991, pp. 313--327.

\bibitem{belytschko-et-al:1994-1}
T.~Belytschko, Y.~Lu, L.~Gu, Element-{F}ree {G}alerkin methods, International
  Journal for Numerical Methods in Engineering 37 (1994) 229--256.

\bibitem{belytschko-et-al:1996-1}
T.~Belytschko, Y.~Krongauz, D.~Organ, M.~Fleming, P.~Krysl, Meshless methods:
  an overview and recent developments, Computer Methods in Applied Mechanics
  and Engineering, special issue 139 (1996) 3--47.

\bibitem{atluri-zhu:1998-1}
S.~Atluri, T.-L. Zhu, A new meshless local {P}etrov-{G}alerkin ({MLPG})
  approach in {C}omputational mechanics, Computational Mechanics 22 (1998)
  117--127.

\bibitem{mirzaei-schaback:2013-1}
D.~Mirzaei, R.~Schaback, Direct {M}eshless {L}ocal {P}etrov-{G}alerkin
  ({DMLPG}) method: a generalized {MLS} approximation, Applied Numerical
  Mathematics 33 (2013) 73--82.

\bibitem{levin:1998-1}
D.~Levin, The approximation power of moving least-squares, Mathematics of
  Computation 67 (1998) 1517--1531.

\bibitem{wendland:2001-1}
H.~Wendland, Local polynomial reproduction and moving least squares
  approximation, IMA Journal of Numerical Analysis 21 (2001) 285--300.

\bibitem{wendland:2000-1}
H.~Wendland, Moving least squares approximation on the sphere, in: T.~Lyche,
  L.~L. Schumaker (Eds.), Mathematical Methods for Curves and Surfaces,
  Nashville, Vanderbilt University Press, Oslo, 2000, pp. 517--526.

\bibitem{wendland:2005-1}
H.~Wendland, Scattered Data Approximation, Cambridge University Press, 2005.

\bibitem{belytschko-et-al:1997-1}
W.~K. Liu, S.~Li, T.~Belytschko, Moving least square reproducing kernel
  methods, ({I}) methodology and convergence, Computer Methods in Applied
  Mechanics and Engineering 143 (1997) 113--154.

\bibitem{armentano-duran:2001-1}
M.~G. Armentano, R.~G. Dur\'{a}n, Error estimates for moving least square
  approximations, Applied Numerical Mathematics 37 (2001) 397--416.

\bibitem{armentano:2001-1}
M.~G. Armentano, Error estimates in {S}obolev spaces for moving least square
  approximations, SIAM Journal on Numerical Analysis 39 (2001) 38--51.

\bibitem{han-meng:2001-1}
W.~Han, X.~Meng, Error analysis of the reproducing kernel particle method,
  Computer Methods in Appleid Mechanics and Engineering 190 (2001) 6157--6181.

\bibitem{zuppa:2003-1}
C.~Zuppa, Error estimates for moving least square approximations, Bulletin of
  the Brazilian Mathematical Society 34 (2003) 231--249.

\bibitem{melenk:2004-1}
J.~M. Melenk, On approximation in meshless methods, in: J.~F. Blowey, A.~W.
  Craig (Eds.), Frontiers of Numerical Analysis, Springer, Durham, 2005, pp.
  65--141.

\bibitem{cheng-cheng:2008-1}
R.~J. Cheng, Y.~M. Cheng, Error estimates for the finite point method, Applied
  Numerical Mathematics 58 (2008) 884--898.

\bibitem{ren-et-al:2014-1}
H.~P. Ren, K.~Y. Pei, L.~P. Wang, Error analysis for moving least squares
  approximation in {2D} space, Applied Mathematics and Computations 238 (2014)
  527--546.

\bibitem{wang-et-al:2014-1}
G.~F. Wang, S.~Y. Hao, Y.~M. Cheng, The error estimates of the interpolating
  element-free {G}alerkin method for two-point boundary value problems,
  Mathematical Problems in Engineering ID 641592 (2014) 12 pages.

\bibitem{mirzaei-et-al:2012-1}
D.~Mirzaei, R.~Schaback, M.~Dehghan, On generalized moving least squares and
  diffuse derivatives, IMA Journal of Numerical Analysis 32 (2012) 983--1000.

\bibitem{brenner-scott:2008-1}
S.~C. Brenner, L.~R. Scott, The Mathematical Theory of Finite Element Methods,
  3nd edition, Springer, 2008.

\bibitem{narcowich-et-al:2004-1}
F.~Narcowich, J.~Ward, H.~Wendland, Sobolev bounds on functions with scattered
  zeros, with application to radial basis function surface fitting, Mathematics
  of Computation 47~(250) (2004) 743--763.

\bibitem{narcowich-et-al:2006-1}
F.~Narcowich, J.~Ward, H.~Wendland, Sobolev error estimates and a {B}ernstein
  inequality for scattered data interpolation via radial basis functions,
  Constructive Approximation 24 (2006) 175--186.

\bibitem{wloka:1987}
J.~Wloka, Partial Differential Equations, Cambridge University Press,
  Cambridge, 1987.

\bibitem{stein:1971}
E.~M. Stein, Singular Integrals and Differentiability Properties of Functions,
  Princeton University Press, Princeton, New Jersey, 1971.

\bibitem{devore-sharpley:1993-1}
R.~A. DeVore, R.~C. Sharpley, Besov spaces on domains in $\mathbb{R}^d$,
  Transactions of the American Mathematical Society 335 (1993) 843--864.

\bibitem{wendland:1999-1}
H.~Wendland, Meshless {G}alerkin methods using radial basis functions,
  Mathematics of Computation 68 (1999) 1521--1531.

\end{thebibliography}
%\bibliographystyle{elsarticle-num}

%%%%%%%%%%%%%%%%%%%%%%%%%%%%%%%%%%%%%%%%%%%%%%%%%%%%%%%%%%%%%%%%%%%%%%%%%%%%%

\end{document}